\documentclass[12pt]{amsart}
\usepackage{amscd,amsmath,amsthm,amssymb,graphics}
\usepackage{pstcol,pst-plot,pst-3d}
\usepackage{lmodern,pst-node}
\usepackage{multicol}
\usepackage{epic,eepic}
\usepackage{amsfonts,amssymb,amscd,amsmath,enumerate,verbatim}
\psset{unit=0.7cm,linewidth=0.8pt,arrowsize=2.5pt 4}

\newpsstyle{fatline}{linewidth=1.5pt}
\newpsstyle{fyp}{fillstyle=solid,fillcolor=verylight}
\definecolor{verylight}{gray}{0.97}
\definecolor{light}{gray}{0.9}
\definecolor{medium}{gray}{0.85}
\definecolor{dark}{gray}{0.6}

\def\frk{\frak}               

\def\Phi{{\frk n}}
\def\Phi{{\frk N}}
\def\opn#1#2{\def#1{\operatorname{#2}}} 
\opn\chara{char} \opn\length{\ell} \opn\pd{pd} \opn\rk{rk}
\opn\projdim{proj\,dim} \opn\injdim{inj\,dim} \opn\rank{rank}
\opn\depth{depth} \opn\grade{grade} \opn\height{height}
\opn\bigheight{bighieght}
\opn\embdim{emb\,dim} \opn\codim{codim}

\opn\Tr{Tr} \opn\bigrank{big\,rank}
\opn\superheight{superheight}\opn\lcm{lcm}
\opn\trdeg{tr\,deg}
\opn\reg{reg} \opn\lreg{lreg} \opn\ini{in} \opn\lpd{lpd}
\opn\size{size}\opn\bigsize{bigsize}
\opn\cosize{cosize}\opn\bigcosize{bigcosize}
\opn\sdepth{sdepth}\opn\sreg{sreg}
\opn\link{link}\opn\fdepth{fdepth}
\opn\div{div} \opn\Div{Div} \opn\cl{cl} \opn\Cl{Cl}
\opn\Spec{Spec} \opn\Supp{Supp} \opn\supp{supp} \opn\Sing{Sing}
\opn\Ass{Ass} \opn\Min{Min}\opn\Mon{Mon} \opn\dstab{dstab} \opn\astab{astab}
\opn\Syz{Syz}
\opn\Ann{Ann} \opn\Rad{Rad} \opn\Soc{Soc}
\opn\Im{Im} \opn\Ker{Ker} \opn\Coker{Coker} \opn\Am{Am}
\opn\Hom{Hom} \opn\Tor{Tor} \opn\Ext{Ext} \opn\End{End}
\opn\Aut{Aut} \opn\id{id}

\opn\nat{nat}
\opn\pff{pf}
\opn\Pf{Pf} \opn\GL{GL} \opn\SL{SL} \opn\mod{mod} \opn\ord{ord}
\opn\Gin{Gin} \opn\Hilb{Hilb}\opn\sort{sort}
\opn\initial{init}
\opn\ende{end}
\opn\height{height}
\opn\type{type}
\opn\aff{aff} \opn\con{conv} \opn\relint{relint} \opn\st{st}
\opn\lk{lk} \opn\cn{cn} \opn\core{core} \opn\vol{vol}
\opn\link{link} \opn\star{star}\opn\lex{lex}
\opn\gr{gr}

\def\pot#1#2{#1[\kern-0.28ex[#2]\kern-0.28ex]}

\opn\dirlim{\underrightarrow{\lim}}
\opn\inivlim{\underleftarrow{\lim}}
\def\Implies{\ifmmode\Longrightarrow \else
        \unskip${}\Longrightarrow{}$\ignorespaces\fi}
\def\implies{\ifmmode\Rightarrow \else
        \unskip${}\Rightarrow{}$\ignorespaces\fi}
\def\iff{\ifmmode\Longleftrightarrow \else
        \unskip${}\Longleftrightarrow{}$\ignorespaces\fi}

\let\:=\colon
\newtheorem{Theorem}{Theorem}[section]
 \newtheorem{Lemma}[Theorem]{Lemma}
 \newtheorem{Corollary}[Theorem]{Corollary}
 \newtheorem{Proposition}[Theorem]{Proposition}
 \newtheorem{Remark}[Theorem]{Remark}
 
 \newtheorem{Example}[Theorem]{Example}
 
 \newtheorem{Definition}[Theorem]{Definition}

\let\epsilon\varepsilon
\let\kappa=\varkappa
\def\qed{\ifhmode\textqed\fi
      \ifmmode\ifinner\quad\qedsymbol\else\dispqed\fi\fi}
\def\textqed{\unskip\nobreak\penalty50
       \hskip2em\hbox{}\nobreak\hfil\qedsymbol
       \parfillskip=0pt \finalhyphendemerits=0}
\def\dispqed{\rlap{\qquad\qedsymbol}}

\opn\dis{dis}
\def\pnt{{\raise0.5mm\hbox{\large\bf.}}}

\opn\Lex{Lex}

\begin{document}
 \title{Sequentially Cohen-Macaulay matroidal ideals}

 \author {Madineh Jafari, Amir Mafi* and Hero Saremi}

\address{M. Jafari, Department of Mathematics, University of Kurdistan, P.O. Box: 416, Sanandaj,
Iran.}
\email{Madineh.jafari3978@gmail.com}

\address{A. Mafi, Department of Mathematics, University of Kurdistan, P.O. Box: 416, Sanandaj,
Iran.}
\email{a\_mafi@ipm.ir}

\address{H. Saremi, Department of Mathematics, Sanandaj Branch, Islamic Azad University, Sanandaj, Iran.} 
\email{hero.saremi@gmail.com}


\keywords{Sequentially Cohen-Macaulay, monomial ideals, matroidal ideals.\\
* Corresponding author}

\begin{abstract}
Let $R=K[x_1,...,x_n]$ be the polynomial ring in $n$ variables over a field $K$ and let $I$ be a matroidal ideal of degree $d$ in $R$. Our main focus is determining
when matroidal ideals are sequentially Cohen-Macaulay. In particular, all sequentially Cohen-Macaulay matroidal ideals of degree $2$ are classified.
Furthermore, we give a classification of sequentially Cohen-Macaulay matroidal ideals of degree $d\geq 3$ in some special cases.
\end{abstract}

\maketitle

\section*{Introduction}
Our goal is to classify the sequentially Cohen-Macaulay matroidal ideals. While for the Cohen-Macaulay property of matroidal ideals, a complete classification was given by Herzog and Hibi \cite{HH2}, the classification of the sequentially Cohen-Macaulay matroidal ideals seems to be much harder. In the present paper partial answers to this problem are given.
Herzog and Hibi \cite{HH1} were the first to give a systematic
treatment of polymatroidal ideals and they studied some combinatoric and algebraic properties related to it. They defined the polymatroidal ideal, a monomial ideal having the exchange property. A square-free polymatroidal ideal is called a matroidal ideal.
Herzog and Takayama \cite{HTa} proved that all polymatroidal ideals have linear quotients which implies that they have linear resolutions.
Herzog and Hibi \cite{HH2}  proved that a polymatroidal ideal $I$ is Cohen-Macaulay (i.e. CM) if and only if $I$ is a principal ideal, a Veronese ideal, or a square-free Veronese ideal.

Let $R=K[x_1,\dots,x_n]$ be the polynomial ring in $n$ indeterminate over a
field $K$ and $I\subset R$ be a homogeneous ideal. For a positive integer $i$, let
$(I_i)$ be the ideal generated by all forms in $I$ of degree $i$. We say that
I is componentwise linear if for each positive integer $i$, $(I_i)$ has a linear
resolution. Componentwise linear ideals were first introduced by Herzog and
Hibi \cite{HH} to generalize Eagon and Reiner's result that the Stanley-Reisner
ideal $I_{\Delta}$ of simplicial complex $\Delta$ has a linear resolution if and only if the Alexander dual $\Delta^{\vee}$ is C.M \cite{ER}. In particular, Herzog and Hibi \cite{HH} and Herzog, Reiner, and Welker \cite{HRW} showed that the Stanley-Reisner ideal $I_{\Delta}$ is componentwise linear if and only if $\Delta^{\vee}$ is sequentially Cohen-Macaulay(i.e. SCM).

It is of interest to understand the SCM matroidal ideals, and this paper
may be considered as a first attempt to characterize such ideals for matroidal ideals
in low degree or in a small number of variables. The remainder of this paper is organized as follows. Section 1 and 2 recall some definitions and results of componentwise linear ideals, simplicial complexes, and polymatroidal ideals.  Section 3 classifies all SCM matroidal ideals of degree $2$. Section 4 studies SCM matroidal ideals of degree $d\geq 3$ over polynomial rings of small dimensional.

For any unexplained notion or terminology, we refer the reader to \cite{HH3} and \cite{V}. Several explicit examples were  performed with help of the computer algebra systems Macaulay2 \cite{GS}.

\section{Preliminaries}
In this section, we recall some definitions and results used throughout
the paper.  As in the introduction, let $K$ be a field and $R=K[x_1,...,x_n]$ be the polynomial ring in $n$ variables over $K$ with each $\deg x_i=1$. Let $I\subset R$ be a monomial ideal and $G(I)$ be its unique minimal set of monomial generators of $I$.

 We say that a monomial ideal $I$ with $G(I)=\{u_1,...,u_r\}$ has {\it linear quotients} if there is an
ordering $\deg(u_1)\leq\deg(u_2)\leq...\leq\deg(u_r)$ such that for each $2\leq i\leq r$ the colon ideal $(u_1,...,u_{i-1}):u_i$ is generated by a subset $\{x_1,...,x_n\}$. 

 It is known that if a monomial ideal $I$ generated in single degree has linear quotients, then $I$ has a linear resolution (see \cite[Lemma 4.1]{CH}). In particular, a monomial ideal $I$ generated in degree $d$ has a linear resolution if and only if the Castelnuovo-Mumford regularity of $I$ is $\reg(I)=d$ (see \cite[Lemma 49]{Va}).
 
 \begin{Lemma} \cite[Corollary 20.19]{E}\label{reg}
 	If $0\longrightarrow A\longrightarrow B\longrightarrow C\longrightarrow 0$ is a short exact sequence of graded finitely generated $R$-modules, then
 	\begin{itemize}
 		\item [(a)]
 		$ \reg A\leq \max (\reg B, \reg C+1).$
 		\item [(b)] $\reg B\leq \max(\reg A,\reg C), $ 
 		\item [(c)]
 		$\reg C\leq \max (\reg A-1, \reg B).$
 		\item [(d)] If $A$ has a finite length, set $s(A)=\max \{s: A_s\neq 0\}$, then $\reg (A)=s(A)$ and the equality holds in $(b)$.
 	\end{itemize}
 \end{Lemma}
 
 One of the important classes of monomial ideals with linear quotients is the class of polymatroid ideals.
 
Let $I\subset R$ be a monomial ideal generated in one degree. We say that $I$ is {\it polymatroidal} if the following "exchange condition" is satisfied: For any two monomials $u=x_1^{a_1}x_2^{a_2}\dots x_n^{a_n}$ and $v=x_1^{b_1}x_2^{b_2}\dots x_n^{b_n}$ belong to $G(I)$ such that $\deg_{x_i}(v)<\deg_{x_i}(u)$, there exists an index $j$ with $\deg_{x_j}(u)<\deg_{x_j}(v)$ such that $x_j(u/x_{i})\in G(I)$. The polymatroidal ideal $I$ is called {\it matroidal} if $I$ is generated by square-free monomials. Note that if $I$ is a matroidal ideal of degree $d$, then $\depth(R/I)=d-1$ (see \cite{C}). 
\begin{Theorem}\cite[Theorem 4.2]{HH2} \label{V}
	A polymatroidal ideal $I$ is CM if and only if $I$ is a principal ideal, a Veronese ideal, or a square-free Veronese ideal.
\end{Theorem}

\section{review on componentwise linear ideals}

For a homogeneous ideal $I$, we write $(I_i)$ to denote the ideal generated by the degree $i$ elements of $I$. Note that $(I_i)$ is different from $I_i$, the vector space of all degree $i$ elements of $I$. Herzog and Hibi
introduced the following definition in \cite{HH}.

\begin{Definition}
	A monomial ideal $I$ is componentwise linear if $(I_i)$ has a linear resolution for all $i$.
	\end{Definition}
A number of familiar classes of ideals are componentwise linear. For example, all ideals with linear resolutions, all stable ideals, all square-free strongly stable ideals are componentwise linear (see \cite{HH3}).

\begin{Proposition}\label{qc}\cite[Proposition 2.6]{FV}
	If $I$is a homogeneous ideal with linear quotients, then $I$ is componentwise linear.
\end{Proposition}
If $I$ is generated by square-free monomials, then we denote by $I_{[i]}$ the ideal generated by the square-free monomials of degree $i$ of $I$.

\begin{Theorem}\cite[Proposition 1.5]{HH}
Let $I$ be a monomial ideal generated by square-free monomials. Then $I$ is componentwise linear if and only if $I_{[i]}$ has a linear resolution for all $i$.
\end{Theorem}

The notion of componentwise linearity is intimately related to the concept
of sequential Cohen-Macaulayness.
\begin{Definition}\cite{S}
	A graded $R$-module $M$ is called {\it sequentially Cohen-Macaulay} (SCM) if there exists a finite filtration of graded $R$-modules
	$0=M_0\subset M_1\subset...\subset M_r=M$ such that each $M_{i}/M_{i-1}$ is Cohen-Macaulay, and the Krull dimensions of the quotients are increasing:
	$$\dim(M_1/M_0)<\dim(M_2/M_1)<...<\dim(M_r/M_{r-1}).$$
	\end{Definition}

The theorem connecting sequentially Cohen-Macaulayness to componentwise
linearity is based on the idea of Alexander duality. We recall the definition of
Alexander duality for square-free monomial ideals and then state the fundamental
result of Herzog and Hibi \cite{HH} and Herzog, Reiner, and Welker \cite{HRW}.

Let $\Delta$ be a simplicial complex on the vertex set $V=\{x_1,x_2,...,x_n\}$, i.e., $\Delta$ is a collection of subsets $V$ such that $(1)$ $\{x_i\}\in \Delta$ for each $i=1,2,...,n$ and $(2)$ if $F\in \Delta$ and $G\subseteq F$, then $G\in\Delta$.
Let $\Delta^{\vee}$ denote the dual simplicial complex of $\Delta$, that is to say, $\Delta^{\vee}=\{V\setminus F\mid F\notin \Delta\}$.

If $I$ is a square-free monomial ideal, then the square-free {\it Alexander dual} of $I=(x_{1,1}...x_{1,n_1},...,x_{t,1}...x_{t,n_t})$ is the ideal  $I^{\vee}=(x_{1,1},...,x_{1,n_1})\cap...\cap(x_{t,1},...,x_{t,n_t}).$

We quote the following results which are proved in \cite{ER}, \cite{HH}, \cite{Te} and \cite{HT}.
\begin{Theorem}\label{T0}
	Let $I$ be a square-free monomial ideal of $R$. Then the following conditions hold:
	\begin{itemize}
		\item[(a)] $R/I$ is CM if and only if the Alexander dual $I^{\vee}$ has a linear resolution.
		\item[(b)] $R/I$ is SCM if and only if the Alexander dual $I^{\vee}$ is componentwise linear.
		\item[(c)] $\projdim(R/I)=\reg(I^{\vee})$.
		\item[(d)] If $y_1,...,y_r$ is an $R$-sequence with $\deg(y_i)=d_i$ and $I=(y_1,...,y_r)$, then $\reg(I)=d_1+...+d_r-r+1$.
	\end{itemize}
\end{Theorem}

In the following if $G(I)=\{u_1,...,u_t\}$, then we set $\supp(I)=\cup_{i=1}^t\supp(u_i)$, where $\supp(u)=\{x_i: u=x_1^{a_1}...x_n^{a_n}, a_i\neq 0\}$. Also we set $\gcd(I)=\gcd(u_1,...,u_m)$ and $\deg(I)=\max\{\deg(u_1),...,\deg(u_m)\}.$

Throughout this paper we assume that all matroidal ideals are full supported, that is, $\supp(I)=\{x_1,...,x_n\}$.

\begin{Corollary}\label{3}\cite[Corollary 6.6]{FV}
	 Let $\Delta$ be a simplicial complex on $n$ vertices, and let $I_{\Delta}$ be it's Stanley-Reisner ideal, minimally generated by square-free monomials $m_1,...,m_s$. If $s\leq3$, so that $\Delta$ has at most three minimal nonfaces, or if $\supp(m_i)\cup\supp(m_j)=\{x_1,...,x_n\}$
	  for all $i\neq j$, then $\Delta$ is SCM.
	\end{Corollary}

\begin{Definition}
	Let $I$ be a monomial ideal of $R$. Then the big height of $I$, denoted by $bight(I)$, is $\max\{\height(\frak{p})| \frak{p}\in\Ass(R/I)\}$.
\end{Definition}

Note that, if $I$ is a matroidal ideal of degree $d$, then by Auslander-Buchasbum formula $bight(I)=n-d+1$.

\begin{Proposition}\cite[Corollary 6.4.20]{V}.\label{P0}
	Let $I$ be a monomial ideal of $R$ such that $R/I$ is SCM. Then $\projdim(R/I)=bight(I)$.
\end{Proposition}

The following examples say that the converse of  Proposition \ref{P0} is not true even if $I$ is matroidal with $\gcd(I)=1$.

\begin{Example}
	Let $n=5$ and $I=(x_1x_2,x_1x_3,x_1x_4,x_1x_5,x_2x_3,x_2x_4,x_3x_5,x_4x_5)$ be an ideal of $R$. Then $I$ is a matroidal ideal of $R$ with $\projdim(R/I)=bight(I)$ but $I$ is not SCM.
\end{Example}
\begin{proof}
	It is clear that $I$ is a matroidal ideal and $$\Ass(R/I)=\{(x_1,x_3,x_4),(x_1,x_2,x_5),(x_2,x_3,x_4,x_5)\}.$$ Thus $I^{\vee}_{[3]}=(x_1x_3x_4,x_1x_2x_5)$ and so $\reg(I^{\vee}_{[3]})=4$. Hence
	$I^{\vee}$ is not componentwise linear resolution. Therefore $I$ is not SCM but $\projdim(R/I)=4=bight(I)$.
\end{proof}

\section{SCM matroidal ideals of degree 2}
In this section, we classify all SCM matroid ideals of degree 2.
\begin{Lemma}\label{L1}
Let $n=3$ and $I$ be a matroidal ideal in $R$ generated in degree $d$. Then $I$ is a SCM ideal.
\end{Lemma}

\begin{proof}
Let $n=3$, then every matroidal ideal in $R$ generated by at most three square-free monomials and so by  Corollary \ref{3} we have the result.
\end{proof}
\begin{Lemma}\label{i}
	Let $I$ be a monomial ideal of $R$ such that $I=(u_1,...,u_d)$ and $\deg(u_i)\leq\deg (u_d)=d$ for all $i$. If $\reg (I)=d$, then $\reg(I_i)=i$ for all $i>d$.
\end{Lemma}
\begin{proof}
	Consider the following exact sequence for $i>d$,
	$$0\longrightarrow\dfrac{I}{(I_i)}\longrightarrow\dfrac{R}{(I_i)}\longrightarrow\dfrac{R}{I}\longrightarrow0.$$

$l(\dfrac{I}{(I_i)})<\infty$, so by Lemma \ref{reg} (d) $\reg(\dfrac{I}{(I_i)})=i-1$ and $$\reg(\dfrac{R}{(I_i)})=\max \{\reg(\dfrac{R}{I}), \reg(\dfrac{I}{(I_i)})\}=\max\{d-1,i-1\}=i-1.$$ On the other hand $\reg(I_i)=\reg(\dfrac{R}{(I_i)})+1$, that is, $\reg (I_i)=i$ for all $i>d$.
\end{proof}
\begin{Proposition}\label{P1}
Let $I$ be monomial ideal which is componentwise linear in $R$. Then $J=(x_{n+1},I)$ is componentwise linear in $R'=K[x_1,...,x_n,x_{n+1}]$.
\end{Proposition}

\begin{proof}
Suppose that $I=(u_1,...,u_m)$, where $\deg(u_i)=d_i$ and $d_{i-1}\leq d_{i}$ for $i=2,...,m$. We induct on $m$, the number of minimal generators of $I$. If $m=1$, then $I=(x_{n+1},u_1)$. Set $J'=x_{n+1}R'$. Note that $(J_j)=(J'_j)$ for all $j<d_1$ and so $(J_j)$ has a linear resolution for all $j<d_1$.
By Theorem \ref{T0}, $\reg(J)=d_1$. Thus $(J_{d_1})$ has a linear resolution and also $(J_j)$ has a linear resolution for all $j>d_1$, by using Lemma \ref{i}.

 Now, let $m>1$ and assume that the ideal $L=(x_{n+1},u_1,...,u_{m-1})$ is componentwise linear. Set $J=(L,u_m)=(I, x_{n+1})$. Note that $(J_j)=(L_j)$ for all $j<d_m$ and so $(J_j)$ has a linear resolution for all $j<d_m$. Hence by using \cite[Lemma 3.2]{HD} we have $\reg(J)=\reg(I)=d_m$. Therefore $(J_{d_m})$ has a linear resolution. Again, by using Lemma \ref{i}, we have $(J_j)$ has a linear resolution for all $j>d_m$. This completes the proof.
\end{proof}

\begin{Corollary}\label{C1}
Let $I$ be a SCM matroidal ideal in $R$ and let $J=x_{n+1}I$ be a monomial ideal in $R'=K[x_1,...,x_n,x_{n+1}]$. Then $J$ is a SCM matroidal ideal in $R'=k[x_1,...,x_n,x_{n+1}]$.
\end{Corollary}

\begin{proof}
The Alexander dual of $J$ is $J^{\vee}=(x_{n+1},I^{\vee})$ and by our hypothesis on $I$, $I^{\vee}$ is componentwise linear resolution. Thus by Proposition \ref{P1}, $J^{\vee}$ is componentwise linear resolution. Thus $J$ is a SCM matroidal ideal of $R'$.
\end{proof}

One of the most distinguished polymatroidal ideals is the ideal of Veronese type. Consider the fixed positive integers $d$ and $1\leq a_1\leq ...\leq a_n\leq d$. The ideal of {\it Veronese type} of $R$ indexed by $d$ and $(a_1,...,a_n)$ is the ideal $I_{(d;a_1,...,a_n)}$ which is generated by those monomials $u=x_1^{i_1}...x_n^{i_n}$ of $R$ of degree $d$ with $i_j\leq a_j$ for each $1\leq j\leq n$.

\begin{Remark}\label{R1}
{\em Let $I$ be a SCM matroidal ideal in $R$ and let $J=x_{n+1}...x_mI$ be a monomial ideal in $R'=K[x_1,...,x_n,x_{n+1},...,x_m]$. Then, by induction on $m$, $J$ is a SCM matroidal ideal in $R'=K[x_1,...,x_n,x_{n+1},...,x_m]$. Hence for a SCM matroidal ideal $J$, we can assume that $\gcd(J)=1$. By using \cite[Lemma 2.16]{KM} all fully supported matroid ideals of degree $n-1(n\geq 2)$ are Veronese type ideals and then by Theorem \ref{V}, all matroidal ideals generated in degrees $d=1,n-1,n$ are SCM.}

\end{Remark}

\begin{Definition}
Let $I$ be a square-free Veronese ideal of degree $d$. We say that $J$ is an almost square-free Veronese ideal of degree $d$ when $J\neq 0$, $G(J)\subseteq G(I)$ and\\ ${\mid G(J)\mid\geq\mid G(I)\mid-1}$. Note that every square-free Veronese ideal is an almost quare-free Veronese ideal. Also, if $J$ is an almost square-free Veronese ideal of degree $n$, then $J$ is a square-free Veronese ideal.
\end{Definition}

\begin{Lemma} \label{L2}
Let $J$ be an almost square-free Veronese ideal of degree $d<n$. Then $J$ is a SCM matroidal ideal of $R$.
\end{Lemma}

\begin{proof}
Suppose that $y_1,...,y_n$ is an arbitrary permutation of the variables of $R$ such that $\{y_1,...,y_n\}=\{x_1,...,x_n\}$ and let $I$ be a square-free Veronese ideal of degree $d$. We may assume that $I=J+(y_{n-d+1}y_{n-d+2}...y_n)$. Then we have $J=(y_1,...,y_{n-d})\cap I$ and so $J$ is a matroidal ideal. Therefore $J^{\vee}=(y_1...y_{n-d},I^{\vee})$. Set $J'=(y_1...y_{n-d})$.  Then, for all $i\leq n-d$, ${J^{\vee}}_{[i]}=J'_{[i]}$ and so it is componentwise linear. For all $i\geq n-d+1$, ${J^{\vee}}_{[i]}$ is a square-free Veronese ideal and so $J^{\vee}$ is a componentwise linear ideal. Hence $J$ is a SCM matroidal ideal, as required.
\end{proof}

From now on, we will let $y_1,...,y_n$ be an arbitrary permutation variables of $R$ such that $\{x_1,...,x_n\}=\{y_1,...,y_n\}$.

\begin{Theorem}\label{T1}
Let $J$ be a matroidal ideal of $R$ with $\deg(J)=2$ and $\gcd(J)=1$. Then $J$ is SCM if and only if there exists a permutation of variables such that the following hold:
\begin{itemize}
\item[(a)] $J=y_1\frak{p}+J'$, where $\frak{p}$ is a monomial prime ideal with $y_1\notin\frak{p}$, $\height(\frak{p})=n-1$ and $J'$ is a SCM matroidal ideal with $\supp(J')=\{y_2,...,y_n\}$ and $\gcd(J')=1$, or
\item[(b)] $J=y_1\frak{p}+y_2\frak{q}$, where $\frak{p}$ and $\frak{q}$ are monomial prime ideals with $y_1\notin\frak{p}$ and $y_1,y_2\notin\frak{q}$ such that $\height(\frak{p})=n-1$, $\height(\frak{q})=n-2$.

\end{itemize}
\end{Theorem}

\begin{proof}
$(\Longleftarrow).$ Consider the case $(a)$.  We have $J=\frak{p}\cap(y_1,J')$, then $J^{\vee}=(\frak{p}^{\vee},y_1J'^{\vee})$ and $\frak{p}^{\vee}\in(u)$ for all $u\in J'^{\vee}$. Since ${J^{\vee}}_{[i]}=y_1{J'^{\vee}}_{[i-1]}$ for all $i\leq n-2$, and ${J'^{\vee}}_{[i-1]}$ is componentwise linear, it follows that ${J^{\vee}}_{[i]}$ is componentwise linear for all $i\leq n-2$. now consider the exact sequence
$$0\longrightarrow R/(y_1J'^{\vee}:\frak{p}^{\vee})(-n+1)\xrightarrow{y_2y_3...y_n} R/y_1J'^{\vee}\longrightarrow R/(\frak{p}^{\vee},y_1J'^{\vee})\longrightarrow 0.$$
From $(y_1J'^{\vee}:\frak{p}^{\vee})=(y_1)$, we have $\reg(R/(y_1J'^{\vee}:\frak{p}^{\vee}))=0$. Since $\deg(\frak{p}^{\vee})=n-1$, we have   $\reg(R/(\frak{p}^{\vee},y_1J'^{\vee}))\geq n-2$. Since $y_1J'^{\vee}$ is componentwise linear and $\deg(u)\leq n-2$ for all $u\in J'^{\vee}$, by \cite[Corollary 8.2.14]{HH3} we have $\reg(R/y_1J'^{\vee})\leq n-2$. By using Lemma \ref{reg},
 \begin{align*}
 \reg(R/(\frak{p}^{\vee},y_1J'^{\vee}))& \leq\max\{\reg(R/(y_1J'^{\vee}:\frak{p}^{\vee})(-n+1))-1,\reg(R/y_1J'^{\vee})\}\\
 &=\max\{n-2,\reg(R/y_1J'^{\vee})\}.
 \end{align*}
  It therefore follows $\reg(R/(\frak{p}^{\vee},y_1J'^{\vee}))=n-2$. Thus ${J^{\vee}}_{[n-1]}$ has a linear resolution and so $J$ is a SCM ideal.
 
Let us consider the case $(b)$. $J=(y_1,y_2)\cap(y_1,\frak{q})\cap\frak{p}$ and so $J^{\vee}=(y_1y_2,y_1\frak{q}^{\vee},\frak{p}^{\vee})$.  It is clear that  $J^{\vee}$ is a monomial ideal with linear quotients. Thus, by Proposition \ref{qc}, $J^{\vee}$ is componentwise linear and so $J$ is a SCM ideal.

$(\Longrightarrow).$ Let $J$ be a SCM ideal. Then there exists $\frak{p}\in\Ass(R/J)$ such that $\height(\frak{p})=\projdim(R/J)=n-1$. Since $J=\cap_{i=1}^n(J:y_i)$ and $\deg(J)=2$, we can consider $\frak{p}=(J:y_1)$ and $\frak{p}=(y_2,...,y_n)$. Hence $J=y_1\frak{p}+J'$, where $J'$ is a matroidal ideal of degree $2$ in $K[y_2,...,y_n]$. We claim that $\supp(J')=\{y_2,...,y_n\}$. Let $y_l\notin\supp(J')$, where $l\geq 2$. Thus $y_1y_l, y_jy_k\in J$, where $j,k\geq 2$. Since $J$ is a matroidal ideal, it follows $y_ly_k$ or $y_ly_j\in J$. Hence $y_ly_k$ or $y_ly_j\in J'$ and this is a contradiction. Therefore
$\supp(J')=\{y_2,...,y_n\}$. $J=\frak{p}\cap(J',y_1)$, it follows that $J^{\vee}=(\frak{p}^{\vee},y_1J'^{\vee})$. For all $i\leq n-2$, we have ${J^{\vee}}_{[i]}=y_1{J'^{\vee}}_{[i-1]}$ and so ${J'^{\vee}}_{[i-1]}$ has a linear resolution for all $i\leq n-2$. Since ${J^{\vee}}_{[n-1]}=y_1{J'^{\vee}}_{[n-2]}+(\frak{p}^{\vee})$ and $\reg({J^{\vee}}_{[n-1]})=n-1$, it follows that $\reg({J'^{\vee}}_{[n-2]})\leq n-2$. Therefore ${J'^{\vee}}_{[n-2]}$ has a linear resolution and so $J'^{\vee}$ is componentwise linear. That is $J'$ is a SCM matroidal ideal of degree $2$. If $\gcd(J')=1$, then $J$ satisfy in the case $(a)$. If $\gcd(J')\neq 1$, then we have the case $(b)$. This completes the proof.
\end{proof}

\section{SCM matroidal ideals over polynomial rings of small dimensional}
 We start this section by the following fundamental lemma.
\begin{Lemma}\label{L3}
Let $n\geq 5$ and $J$ be a matroidal ideal of degree $d$ in $R$ and $\gcd(J)=1$. If $J$ is SCM, then
\begin{align*}
J=&y_1y_2...y_{d-1}\frak{p}+y_1y_2...y_{d-2}J_1+y_1y_2...y_{d-3}y_{d-1}J_2+...+\\
&y_1y_3...y_{d-1}J_{d-2}+y_2y_3...y_{d-1}J_{d-1}+J_d,
\end{align*}
 where $\frak{p}=(y_d,...,y_n)$ is a monomial prime ideal, $J_i$ is a SCM matroidal ideal of degree $2$ with $\supp(J_i)=\{y_d,y_{d+1},...,y_n\}$ for $i=1,...,d-1$ and $J_d\subseteq\cap_{i=1}^{d-1}J_i$.
\end{Lemma}

\begin{proof}
 $J$ is a SCM matroidal ideal, then there is a prime ideal $\frak{p}\in\Ass(R/J)$ such that $\height(\frak{p})=\projdim(R/J)$. Since $\depth(R/J)=d-1$, it follows that $\height(\frak{p})=n-d+1$. For every square-free monomial ideal in $R$, we have $J=\cap_{i=1}^n(J:y_i)$. It follows that $\frak{p}=(J:y_1y_2...y_{d-1})$ and we can write $J=y_1...y_{d-1}\frak{p}+J'$, where $J'$ is a square-free monomial ideal of degree $d$. It is clear that $J'$ has a presentation $$J'=y_1y_2...y_{d-2}J_1+y_1y_2...y_{d-3}y_{d-1}J_2+...+y_1y_3...y_{d-1}J_{d-2}+y_2y_3...y_{d-1}J_{d-1}+J_d$$ and $J_d\subseteq\cap_{i=1}^{d-1}J_i$. Note that $\gcd(J)=1$ and $$(J:y_1y_2...y_{d-i-1}y_{d-i+1}...y_{d-1})=y_{d-i}\frak{p}+J_i,$$ we have $\height(J)\geq 2$ and so $J_i\neq 0$ for $i=1,...,d-1$. It is known that the localization of every SCM ideal is SCM and so $$(J:y_1y_2...y_{d-i-1}y_{d-i+1}...y_{d-1})=y_{d-i}\frak{p}+J_i$$ is a SCM matroidal ideal of degree $2$ for $i=1,...,d-1$. By using the proof of Theorem \ref{T1}, $J_i$ is a SCM matroidal ideal with $\supp(J_i)=\{y_d,y_{d+1},...,y_n\}$ for $i=1,...,d-1$.
\end{proof}

It is known that the localization of each SCM matroidal ideal is a SCM matroidal ideal. The following example shows that the converse is not true.
\begin{Example}
Let $n=4$ and $J=(x_1x_3,x_1x_4,x_2x_3,x_2x_4)$. Then $J$ is a matroidal ideal and $(J:x_i)$ is SCM matroidal for $i=1,2,3,4$; but $J$ is not SCM.
\end{Example}

\begin{proof}
It is clear that $J$ is matroidal and $(J:x_i)$ is SCM matroidal for $i=1,2,3,4$. Since $J^{\vee}=(x_1x_2,x_3x_4)$, it follows that $\reg(J^{\vee})=3$. Therefore $J$ is not SCM.
\end{proof}

From now on, as Lemma \ref{L3}, for a SCM matroidal ideal $J$ of degree $d$ and $\gcd(J)=1$ in $R$ with $n\geq 5$, we can write $$J=y_1...y_{d-1}\frak{p}+y_1y_2...y_{d-2}J_1+y_1y_2...y_{d-3}y_{d-1}J_2+...+y_2y_3...y_{d-1}J_{d-1}+J_d,$$where $\frak{p}=(y_d,...,y_n)$ is a monomial prime ideal, $J_i$ is a SCM matroidal ideal of degree $2$ with $\supp(J_i)=\{y_d,y_{d+1},...,y_n\}$ for $i=1,...,d-1$ and $J_d\subseteq\cap_{i=1}^{d-1}J_i$.

 Note that if for instance $\gcd(J_1)=y_d$, then we have $$J=y_1...y_{d-1}\frak{p}+y_1y_2...y_{d-2}y_d\frak{q}+y_1y_2...y_{d-3}y_{d-1}J_2+...+y_1y_3...y_{d-2}J_{d-2}+y_2y_3...y_{d-1}J_{d-1}+J_d,$$ where $\frak{q}=(y_{d+1},...,y_n)$.

Bandari and Herzog in \cite[Proposition 2.7]{BH} proved that if $n=3$ and $J$ is a matroidal ideal with $\gcd(J)=1$, then $J$ is a square-free Veronese ideal and so by Theorem \ref{V}, it is CM (see also \cite[Proposition 1.5]{MN}). In the following proposition we prove this result in the case $n=4$ for SCM ideals.

\begin{Proposition}\label{P2}
Let $n=4$ and $J$ be a matroidal ideal of $R$ of degree $d$ and $\gcd(J)=1$. Then $J$ is a SCM ideal if and only if $J$ is
\begin{itemize}
\item[(a)] a square-free Veronese ideal, or
\item[(b)] an almost square-free Veronese ideal.
\end{itemize}
\end{Proposition}
\begin{proof}
$(\Longleftarrow)$ is clear by Theorem \ref{V} and Lemma \ref{L2}.

$(\Longrightarrow)$. If $d=1,3, 4$, then by Theorem \ref{V} and \cite[Lemma 2.16]{KM} $J$ is a square-free Veronese ideal.
If $d=2$, then by Theorem \ref{T1}, $J=y_1\frak{p}+J'$, where $\frak{p}$ is a monomial prime ideal with $y_1\notin\frak{p}$, $\height(\frak{p})=3$ and $J'$ is a SCM matroidal ideal with $\supp(J')=\{y_2,y_3,y_4\}$. If $\gcd(J')=1$, then $J'$ is a square-free Veronese ideal and so is $J$.
If $\gcd(J')\neq 1$, then $J'$ is an almost square-free Veronese ideal.
\end{proof}
\begin{Proposition}
Let $n=4$ and $J$ be a matroidal ideal of $R$ of degree $d$. Then $J$ is a SCM ideal if and only if $\projdim(R/J)=bight(J)$.
\end{Proposition}

\begin{proof}
$(\Longrightarrow)$. It follows by Proposition \ref{P0}.

$(\Longleftarrow)$. If $d=1,3,4$, then by Remark \ref{R1} $J$ is SCM.
Let $d=2$. By our hypothesis, there exists $\frak{p}\in\Ass(R/J)$ such that $\frak{p}=(J:y_1)$. Thus $J=y_1\frak{p}+J'$, where $J'$ is matroidal ideal of degree $2$ in $K[y_2,y_3,y_4]$. Hence $J'$ is a square-free Veronese ideal or an almost square-free Veronese ideal. Therefore by Proposition \ref{P2}, $J$ is SCM.
\end{proof}

\begin{Lemma}\label{L4}
Let $n\geq 5$ and $J$ be a matroidal ideal of degree $3$ in $R$ such that $J=y_1y_2\frak{p}+y_1y_3\frak{q}+y_2y_3\frak{q}$, where $\frak{p}$ and $\frak{q}$ are monomial prime ideals with $y_1,y_2\notin\frak{p}$ and $y_1,y_2,y_3\notin\frak{q}$ such that $\height(\frak{p})=n-2$, $\height(\frak{q})=n-3$. Then $J$ is SCM.
\end{Lemma}

\begin{proof}
Since $J=\frak{p}\cap(y_1y_2,y_1y_3\frak{q},y_2y_3\frak{q})$, it follows that $J=\frak{p}\cap(y_1,y_2)\cap(y_1,y_3)\cap(y_2,y_3)\cap(y_1,\frak{q})\cap(y_2,\frak{q})$.
Therefore $J^{\vee}=(y_1y_2,y_1y_3,y_2y_3,y_1{\frak{q}}^{\vee},y_2{\frak{q}}^{\vee},\frak{p}^{\vee})$. It is clear that $J^{\vee}$ is a monomial ideal with linear quotients and so by Proposition \ref{qc}, $J^{\vee}$ is componentwise linear. Thus $J$ is SCM.
\end{proof}

\begin{Lemma}\label{L5}
Let $n\geq 5$ and $J$ be a matroidal ideal of degree $3$ such that $$J=y_1y_2\frak{p}+y_1y_3\frak{q}_1+y_2y_4\frak{q}_2+J_1,$$ where $\frak{p}$, $\frak{q}_1$ and $\frak{q}_2$ are monomial prime ideals with $y_1,y_2\notin\frak{p}$, $y_1,y_2,y_3\notin\frak{q}_1$ and $y_1,y_2,y_4\notin\frak{q}_2$ such that $\height(\frak{p})=n-2$, $\height(\frak{q}_1)=n-3=\height(\frak{q}_2)$ and $J_1$ is a matroidal ideal in $R'=K[y_3,...,y_n]$. Then $G(J_1)=\{y_3y_4y_i\mid i=5,6,...,n\}$. In particular, $J$ is not SCM.
\end{Lemma}
\begin{proof}
We consider two cases:
\begin{description}
	\item[ Case (a)] $J_1=0$, then we have  $y_1y_3y_5, y_2y_3y_4\in J$ but $y_2y_3y_5$ or $y_3y_4y_5$ are not elements of $J$. Thus $J$ is not a matroidal ideal and this is a contradiction.
	
	\item[ Case (b)]
	$J_1\neq 0$. 
	\begin{description}
		\item[1)]For $n=5$,  $J_1=(y_3y_4y_5)$ and
		\begin{equation*}
		J=(y_1,y_4)\cap(y_2,y_3)\cap(y_1,y_2,J_1)\cap(y_2,y_3,J_1)\cap(y_1,\frak{q}_2)\cap(y_2,\frak{q}_1)\cap\frak{p}.
		\end{equation*}
		Therefore $\reg(J^{\vee}_{[2]})=3$ and so $J$ is not SCM. 
		\item[2)] Suppose that $n\geq 6$. Then $(J:y_3)=(y_1y_2,y_2y_4,y_1\frak{q}_1,(J_1:y_3))$. If $y_iy_j\in (J:y_3)$ for $5\leq i\neq j\leq n$, then $y_2y_i\in(J:y_3)$ for $i\geq 5$, since $y_2y_4\in (J:y_3)$. But this is a contradiction. Therefore $y_3y_iy_j\notin J$ for all $5\leq i\neq j\leq n$. Consider $(J:y_4)$, we have
		$y_4y_iy_j\notin J$ for all $5\leq i\neq j\leq n$. Also, if $y_iy_jy_t\in J$ for different numbers $i,j,t$ with $5\leq i,j, t\leq n$, then since $y_1y_3y_i\in J$, we have $y_3y_iy_j\in J$ or $y_3y_iy_t\in J$ and this is a contradiction. Thus $G(J_1)\subseteq\{y_3y_4y_i\mid i=5,6,...,n\}$. On the other hand, since $y_2y_4y_i$ and $y_1y_3y_i$ are elements in $J$ for $i\geq 5$ we have $y_3y_4y_i\in J$ for $i\geq 5$. Hence $G(J_1)=\{y_3y_4y_i\mid i=5,6,...,n\}$.
		Therefore $$J=(y_1,y_4)\cap(y_2,y_3)\cap(y_1,y_2,J_1)\cap(y_1,\frak{q}_2)\cap(y_2,\frak{q}_1)\cap\frak{p}$$ and so $J^{\vee}=(y_1y_4,y_2y_3,y_1y_2J_1^{\vee},y_1\frak{q}_2^{\vee},y_2\frak{q}_1^{\vee},\frak{p}^{\vee})$. Thus $\reg(J^{\vee}_{[2]})=3$ and so $J$ is not SCM.
	\end{description} 
\end{description}

\end{proof}

\begin{Lemma}\label{L6}
Let $n\geq 6$ and $J$ be a matroidal ideal of degree $3$ such that $J=y_1y_2\frak{p}+y_1y_3\frak{q}+y_2J_1$, where $\frak{p}$ and $\frak{q}$ are monomial prime ideals with $y_1,y_2\notin\frak{p}$, $y_1,y_2,y_3\notin\frak{q}$ such that $\height(\frak{p})=n-2$, $\height(\frak{q})=n-3$ and $J_1$ is a matroidal ideal in $R'=K[y_3,...,y_n]$ with $\gcd(J_1)=1$. Then $J$ is not SCM matroidal.
\end{Lemma}
\begin{proof}
By contrary, we assume that $J$ is SCM matroidal. Then $(J:y_2)=y_1\frak{p}+J_1$ is SCM matroidal and so by Theorem \ref{T1} $J_1$ is SCM matroidal of degree $2$.
From $\gcd(J_1)=1$, we have $J_1=y_i\frak{q}_1+J_2$, where $\frak{q}_1$ and $J_2$ are a monomial prime ideal of height $n-3$ and a matroidal ideal respectively in $R'=K[y_3,...,y_{i-1},y_{i+1},...,y_n]$. There are two main cases to consider.
\begin{description}
	\item[a)] $i=3$, then $(J:y_j)=(y_1y_2,y_1y_3,y_2y_3,y_2(J_2:y_j))$ when $j\neq 1,2,3$. Since $y_t\in(J_2:y_j)$ for $t\neq 1,2,3,j$, we have $y_2y_t$ and $y_1y_3$ are elements of $(J:y_j)$ but $y_1y_t$ or $y_3y_t$ are not elements of $(J:y_j)$. This is a contradiction.
	\item[b)]
	 $i\neq 3$, then $(J:y_i)=(y_1y_2,y_1y_3,y_2\frak{q}_1)$. Thus $y_2y_t$ and $y_1y_3$ for $t\neq 3$ are elements of $(J:y_i)$ but $y_1y_t$ or $y_3y_t$ are not elements of $(J:y_i)$ and this is a contradiction. Thus $J$ is not SCM matroidal.
\end{description}

\end{proof}

\begin{Lemma}\label{L7}
Let $n\geq 6$ and $J$ be a matroidal ideal of degree $3$ such that $J=y_1y_2\frak{p}+y_1y_3\frak{q}+y_2J_1+J_2$ or $J=y_1y_2\frak{p}+y_1y_3\frak{q}+y_2y_3\frak{q}+J_2$, where $\frak{p}$ and $\frak{q}$ are monomial prime ideals with $y_1,y_2\notin\frak{p}$, $y_1,y_2,y_3\notin\frak{q}$ such that $\height(\frak{p})=n-2$, $\height(\frak{q})=n-3$ and $J_1$ is a nonzero matroidal ideal in $R'=K[y_3,...,y_n]$ with $\gcd(J_1)=1$. Then $G(J_2)\subseteq\{y_3y_iy_j\mid  4\leq i\neq j\leq n\}$ and if $J_2\neq 0$, then $\supp(J_2)=\{y_3,y_4,...,y_n\}$. In particular, if $J=y_1y_2\frak{p}+y_1y_3\frak{q}+y_2y_3\frak{q}+J_2$, then $J_2=0$.
\end{Lemma}

\begin{proof}
Let us consider $J=y_1y_2\frak{p}+y_1y_3\frak{q}+y_2J_1+J_2$. Then we have $(J:y_t)=(y_1y_2,y_1y_3,y_2(J_1:y_t),(J_2:y_t))$ for some $t\geq 4$. If $y_iy_jy_t\in J$ for some different numbers $4\leq i,j, t\leq n$, then $y_iy_j\in(J:y_t)$. Since $y_1y_3\in (J:y_t)$, it follows that $y_1y_i\in (J:y_t)$ for some $i\geq 4$ and this is a contradiction. It therefore follows that $G(J_2)\subseteq\{y_3y_iy_j\mid  4\leq i\neq j\leq n\}$. Also,
$(J:y_3)=(y_1y_2,y_1\frak{q},y_2(J_1:y_3),(J_2:y_3))$. If $y_iy_j\in(J:y_3)$ for some $4\leq i\neq j\leq n$, then $y_iy_t\in(J:y_3)$ for all $t$ with $4\leq i\neq t\leq n$ since $y_1y_t\in(J:y_3)$. Hence $\supp(J_2)=\{y_3,y_4,...,y_n\}$.
The proof for the case $J=y_1y_2\frak{p}+y_1y_3\frak{q}+y_2y_3\frak{q}+J_2$ is similar to the above argument.
In particular, if $y_3y_iy_j\in J_2$ for some  $4\leq i\neq j\leq n$ then from $y_1y_2y_t\in J$ for some $4\leq i\neq t\neq j\leq n$ we have $y_iy_jy_t\in J$. This is a contradiction. Thus $J_2=0$.
\end{proof}

\begin{Proposition}\label{P3}
Let $n=5$ and $J$ be a matroidal ideal of degree $3$ such that $\gcd(J)=1$. Then $J$ is a SCM ideal if and only if $J=y_1y_2\frak{p}+y_1J_1+y_2J_2+J_3$, where $J_1$ and $J_2$ are SCM ideals with $\supp(J_1)=\supp(J_2)=\{y_3,y_4,y_5\}$, $J_3\subseteq J_1\cap J_2$ and satisfying in the one of the following cases:
\begin{itemize}
\item[(a)] $\gcd(J_1)=1$, $\gcd(J_2)=1$, or
\item[(b)] $\gcd(J_1)=y_3=\gcd(J_2)$ and $J_3=0$.
\end{itemize}
\end{Proposition}

\begin{proof}
$(\Longleftarrow)$. Consider $(a)$. Then $J_1$ and $J_2$ are square-free Veronese ideal and $G(J_3)\subseteq\{y_3y_4y_5\}$. If $J_3=0$,
then $J$ is an almost square-free Veronese ideal and so by using Lemma \ref{L2}, $J$ is a SCM matroidal ideal.
If $J_3\neq 0$, then $J$ is a square-free Veronese ideal and so $J$ is a SCM matroidal ideal.

If we have the case $(b)$, then by Lemma \ref{L4} the result follows.

$(\Longrightarrow).$ Let $J$ be a SCM, then by Lemma \ref{L3}, $J$ has the presentation $J=y_1y_2\frak{p}+y_1J_1+y_2J_2+J_3$, where $J_1$ and $J_2$ are SCM matroidal ideals with $\supp(J_1)=\supp(J_2)=\{y_3,y_4,y_5\}$ and $J_3\subseteq J_1\cap J_2$.
\begin{description}
	\item[1)] If $\gcd(J_1)=y_3$ and $\gcd(J_2)=y_4$, then by Lemma \ref{L5} $J$ is not a SCM matroidal ideal and we don't have this case. 
	\item[2)] If $\gcd(J_1)=\gcd(J_2)=y_3$, then $J_3=0$. Let contrary, then $G(J_3)=\{y_3y_4y_5\}$ and $y_1y_2y_5,y_3y_4y_5\in J$ but $y_1y_4y_5$ or $y_2y_4y_5$ are not elements of $J$. This is a contradiction.
	\item[3)] If  $\gcd(J_1)=y_3$, $\gcd(J_2)=1$ and $J_3=0$, then $y_1y_3y_5, y_2y_4y_5\in J$ but $y_1y_4y_5$ or $y_3y_4y_5$ are not elements of $J$. Therefore $J$ is not matroidal and we don't have this case.
	\item[4)] If $\gcd(J_1)=y_3$, $\gcd(J_2)=1$ and $G(J_3)=\{y_3y_4y_5\}$, then by change of variables $(a)$ follows with $J_3=0$. 
\end{description}

\end{proof}

\begin{Proposition}\label{P4}
Let $n=6$ and let $J$ be a matroidal ideal of degree $4$ such that $\gcd(J)=1$. Then $J$ is a SCM ideal if and only if  $J=y_1y_2y_3\frak{p}+y_1y_2J_1+y_1y_3J_2+y_2y_3J_3+J_4$ such that $J_1, J_2,J_3$ are SCM matroidal ideals and satisfying in one of the following conditions:
\begin{itemize}
\item[(a)] for $i=1,2,3$, $\gcd(J_i)=1$ and $\mid G(J_4)\mid=3$,
\item[(b)] for $i=1,2,3$, $\gcd(J_i)=1$ and $\mid G(J_4)\mid=2$,
\item[(c)] for $i=1,2,3$, $\gcd(J_i)=1$ and $J_4=0$, or
\item[(d)] for $i=1,2,3$, $\gcd(J_i)=y_4$ and $J_4=0$.
\end{itemize}
\end{Proposition}

\begin{proof}
$(\Longleftarrow).$ If we have $(a)$, then $J$ is a square-free Veronese ideal and so by Theorem \ref{V}, $J$ is SCM. Consider case $(b)$, then $J$ is an almost square-free Veronese ideal and so by Lemma \ref{L2}, $J$ is SCM. If we consider $(d)$, then by using the same proof of Lemma \ref{L4} $J^{\vee}$ has linear quotients and so $J$ is SCM.
Let $(c)$, then we have $J=\frak{p}\cap(y_1,y_2)\cap(y_1,y_3)\cap(y_2,y_3)\cap(y_1,J_3)\cap(y_2,J_2)\cap(y_3,J_1)$ and so $J^{\vee}=(y_1y_2,y_1y_3,y_2y_3,y_1J^{\vee}_3,y_2J^{\vee}_2,y_3J^{\vee}_1,\frak{p}^{\vee})$. That is, $J^{\vee}$ has linear quotients. Thus $J$ is SCM.

$(\Longrightarrow).$ Let $J$ be a SCM ideal. Then by Lemma \ref{L3}, $J=y_1y_2y_3\frak{p}+y_1y_2J_1+y_1y_3J_2+y_2y_3J_3+J_4$ and $J_1, J_2,J_3$ are SCM matroidal ideals. Let $\gcd(J_1)=y_4$. Since $(J:y_1)=y_2y_3\frak{p}+y_2J_1+y_3J_2+(J_4:y_1)$, $\gcd(J:y_1)=1$ and $(J:y_1)$ is a SCM matroidal ideal, by Proposition \ref{P3} it follows $\gcd(J_2)=y_4$ and $(J_4:y_1)=0$. Again by using $(J:y_2)$ and $(J:y_3)$, we obtain $\gcd(J_1)=\gcd(J_3)=\gcd(J_2)=y_4$ and $J_4=0$. Also, if for some $i$, $\gcd(J_i)=1$, then by Proposition \ref{P3} and by using $(J:y_1), (J:y_2)$ and $(J:y_3)$ we have $\gcd(J_i)=1$ for $i=1,2,3$. If $G(J_4)=\{y_1y_4y_5y_6\}$, then $J$ is not a matroidal ideal since $y_1y_4y_5y_6, y_2y_3y_5y_6\in J$, but $y_2y_4y_5y_6$ or $y_3y_4y_5y_6$ are not elements of $J$. Thus $J_4=0$ or $\mid G(J_4)\mid=2$ or $\mid G(J_4)\mid=3$ and this completes the proof.
\end{proof}

\begin{Proposition}\label{P5}
Let $n\geq 6$ and let $J$ be a matroidal ideal of degree $n-2$ such that $\gcd(J)=1$. Then $J$ is a SCM ideal if and only if
\begin{align*}
J=&y_1y_2...y_{n-3}\frak{p}+y_1y_2...y_{n-4}J_1+y_1y_2...y_{n-5}y_{n-3}J_2+...+\\
&y_1y_3...y_{n-3}J_{n-4}+y_2y_3...y_{n-3}J_{n-3}+J_{n-2}
\end{align*} 
such that $J_i$ is SCM matroidal ideal for all $i=1,..,n-3$ and satisfying in one of the following conditions:
\begin{itemize}
\item[(a)] for $i=1,...,n-3$, $\gcd(J_i)=1$ and $\mid G(J_{n-2})\mid=\binom{n-3}{2}$,
\item[(b)] for $i=1,...,n-3$, $\gcd(J_i)=1$ and $\mid G(J_{n-2})\mid=\binom{n-3}{2}-1$,
\item[(c)] for $i=1,...,n-3$, $\gcd(J_i)=1$ and $J_{n-2}=0$, or
\item[(d)] for $i=1,...n-3$,  $\gcd(J_i)=y_{n-2}$ and $J_{n-2}=0$.
\end{itemize}
\end{Proposition}

\begin{proof}
$(\Longleftarrow).$ 

If case $(a)$ holds, then $J$ is a square-free Veronese ideal and so by Theorem \ref{V}, $J$ is SCM.
 Let $(b)$, then $J$ is an almost square-free Veronese ideal and so by Lemma \ref{L2}, $J$ is SCM. 
  If $(d)$, then by using the same proof of  Lemma \ref{L4}, $J^{\vee}$ has linear quotients and so $J$ is SCM.
Let $(c)$, then we have
\begin{align*}
J=&\frak{p}\cap(y_1,y_2)\cap...\cap(y_1,y_{n-3})\cap(y_2,y_3)\cap...\cap(y_2,y_{n-3})\\
&\cap...\cap(y_{n-4},y_{n-3})\cap(y_1,J_{n-3})\cap...\cap(y_{n-3},J_1)
\end{align*} 
 and so $$J^{\vee}=(y_1y_2,...,y_1y_{n-3},y_2y_3,...,y_2y_{n-3},...,y_{n-4}y_{n-3},y_1J^{\vee}_{n-3},...,y_{n-3}J^{\vee}_1,\frak{p}^{\vee}).$$ Since $J_i$ are square-free Veronese ideals, it follows that $J^{\vee}$ has linear quotients. That is, $J$ is SCM.

$(\Longrightarrow).$ Let $J$ be a SCM ideal. Then by Lemma \ref{L3},
\begin{align*}
J=&y_1y_2...y_{n-3}\frak{p}+y_1y_2...y_{n-4}J_1+y_1y_2...y_{n-5}y_{n-3}J_2+...+\\
&y_1y_3...y_{n-3}J_{n-4}+y_2y_3...y_{n-3}J_{n-3}+J_{n-2}
\end{align*}
 and $J_i$ are SCM matroidal ideals for all $i=1,...,n-3$. We use induction on $n\geq 6$. If $n=6$, then the result follows by Proposition \ref{P4}. Let $n>6$ and $\gcd(J_1)=y_{n-2}$. 
 $$(J:y_1)=y_2y_3...y_{n-3}\frak{p}+y_2...y_{n-4}J_1+y_2...y_{n-5}y_{n-3}J_2+...+y_3...y_{n-3}J_{n-4}+(J_{n-2}:y_1),$$ $\gcd(J:y_1)=1$ and $(J:y_1)$ is a SCM matroidal ideal, by induction hypothesis it follows $\gcd(J_i)=y_{n-2}$ for $i=1,...,n-4$ and $(J_{n-2}:y_1)=0$. Again by using $(J:y_i)$ for $i=2,...,n-3$ and by using induction hypothesis, $\gcd(J_i)=y_{n-2}$ for $i=1,...n-3$ and $J_{n-2}=0$. Also, if for some $i$, $\gcd(J_i)=1$, then again by using $(J:y_i)$ for $i=1,...,n-3$ and by using induction hypothesis we have $\gcd(J_i)=1$ for $i=1,...,n-3$. If $\mid G(J_{n-2})\mid<\binom{n-3}{2}-1$, then there exists $1\leq i\leq{n-3}$ such that $\mid G(I:y_i)\mid<\binom{n-4}{2}-1$ and this is a contradiction. Thus $J_{n-2}=0$ or $\mid G(J_{n-2})\mid=\binom{n-3}{2}$ or $\mid G(J_{n-2})\mid=\binom{n-3}{2}-1$ and this completes the proof.
\end{proof}

\begin{Theorem}\label{T2}
Let $n=6$ and let $J$ be a matroidal ideal of degree $3$ such that $\gcd(J)=1$. Then $J$ is a SCM ideal if and only if $J=y_1y_2\frak{p}+y_1J_1+y_2J_2+J_3$ such that $J_1$ and $J_2$ are SCM matroidal ideals and satisfying in one of the following conditions:
\begin{itemize}
\item[(a)] $\mid G(J_3)\mid=4$ and one of $J_1$ or $J_2$ is an almost square-free Veronese ideal and the other is a square-free Veronese ideal,
\item[(b)] $\mid G(J_3)\mid=3$, $J_1$, $J_2$ are square-free Veronese ideals,
\item[(c)] $J_3=0$ and  $J_1=J_2$ are square-free Veronese ideals or almost square-free Veronese ideals either $J_3=0$ and $\gcd(J_1)=y_3=\gcd(J_2)$.
\end{itemize}
\end{Theorem}
\begin{proof}
$(\Longleftarrow)$. If we consider the $(a)$ or $(b)$, then $J$ is a square-free Veronese ideal or an almost square-free Veronese ideal and so  $J$ is SCM. Consider $(c)$ and suppose that $\gcd(J_1)=\gcd(J_2)=y_3$. Then by using Lemma \ref{L4}, $J$ is SCM. Also, for $(c)$ if $J_1=J_2$ are square-free Veronese ideals or almost square-free Veronese ideals, we have $J^{\vee}=(y_1y_2,y_1J^{\vee}_2,y_2J^{\vee}_1,\frak{p}^{\vee})$ and so $J^{\vee}$ has linear quotients. Thus $J$ is SCM.

$(\Longrightarrow).$ Let $J$ be a SCM ideal. Then by Lemma \ref{L3}, $J=y_1y_2\frak{p}+y_1J_1+y_2J_2+J_3$ and $J_1$ and $J_2$ are SCM matroidal ideals and $J_3\subseteq J_1\cap J_2$ with $\suppressfloats(J_3)=\{y_3,y_4,y_5,y_6\}$.
Therefore $\mid G(J_3)\mid\leq4$. We have four cases:
\begin{description}
	\item[Case (i)]
 Suppose that $\mid G(J_3)\mid=4$ ,then by Lemmas \ref{L5} and \ref{L7} we have $\gcd(J_1)=1=\gcd(J_2)$. By Proposition \ref{P2}, we have the case $(a)$ if we prove $J_1$ and $J_2$ aren't  almost square-free Veronese ideals in the same time. Let contrary, if $y_1y_3y_5,y_2y_3y_5$ are not elements of $J$, then $y_1y_2y_3,y_3y_4y_5\in J$. But $y_1y_3y_5$ or $y_2y_3y_5$ are not elements of $J$ and this is a contradiction. 

If $y_1y_3y_5, y_2y_3y_6$ are not elements of $J$, then
\begin{align*}
(J:y_3)=&(y_1y_2,y_1(y_4,y_6),y_2(y_4,y_5),y_4y_5,y_4y_6,y_5y_6)\\
=&y_4(y_1,y_2,y_5,y_6)+(y_1y_2,y_1y_6,y_2y_5,y_5y_6).
\end{align*}
By Theorem \ref{T1}, $(y_1y_2,y_1y_6,y_2y_5,y_5y_6)$ is not SCM and this is a contradiction. 

If $y_1y_3y_5, y_2y_4y_6$ are not elements of $J$, then
$(J^{\vee}_{[3]})=(y_1y_3y_5, y_2y_4y_6)$ and so $\reg(J^{\vee}_{[3]})=5$. Thus $J$ is not SCM and this is a contradiction. 
\item[Case (ii)] Let $\mid G(J_3)\mid=3$. We consider the following cases.
\begin{description}
	\item[1)] If $\gcd(J_1)=y_3$ and $\gcd(J_2)=1$, then 
	$G(J_3)=\{y_3y_4y_5,y_3y_4y_6,y_3y_5y_6\}$, by Lemma \ref{L7}.
	$\gcd(J_2)=1$, so by Proposition \ref{P2}, $J_2$ is a square-free Veronese ideal or an almost square-free Veronese ideal. If $J_2=(y_2y_3y_4,y_2y_3y_5,y_2y_3y_6,y_2y_4y_5,y_2y_4y_6)$ is an almost square-free Veronese ideal, then $y_3y_5y_6,y_1y_2y_4\in J$ but $y_1y_5y_6$ or $y_2y_5y_6$ either $y_4y_5y_6$ are not elements of $J$ and this is a contradiction. So $J_2$ is a square-free Veronese ideal and by using a new presentation for $J$ and change of variables we get $J_1$ and $J_2$ are square-free Veronese ideals and $J_3=0$ and this is the case $(c)$.
	
	\item[2)]
	If $\gcd(J_1)=y_3$ and $\gcd(J_2)=y_4$, then by Lemma \ref{L5} we have $\mid G(J_3)\mid=2$ and this is a contradiction.

	\item[3)] If $\gcd(J_1)=y_3=\gcd(J_2)$, then $y_1y_2y_4,y_3y_4y_5\in J$ but $y_1y_4y_5$ or $y_2y_4y_5$ are not elements of $J$ and this is a contradiction.
	
	\item[4)] Let $\gcd(J_1)=1=\gcd(J_2)$. Suppose that $J_1$ is a square-free Veronese ideal and $J_2$ is an almost square-free Veronese ideal. We assume that $J_2=(y_2y_3y_4,y_2y_3y_5,y_2y_3y_6,y_2y_4y_5,y_2y_4y_6)$. Since $\mid G(J_3)\mid=3$, we can assume that one of the element $y_3y_5y_6$ or $y_3y_4y_6$ are not in $J$. If $y_3y_5y_6\notin J$, then $y_2y_3y_5,y_1y_5y_6\in J$ but $y_2y_5y_6$ or $y_3y_5y_6$ are not elements of $J$ and this is a contradiction.
	If $y_3y_4y_6\notin J$, then $(J:y_6)=(y_1y_2,y_1(y_3,y_4,y_5),y_2(y_3,y_4),y_3y_5,y_4y_5)$. Therefore by using Theorem \ref{T1} this is not SCM. Thus we do not have this case.
	Also, by the same argument of the {\bf Case (i)}, $J_1$ and $J_2$ are not almost square-free Veronese ideals in the same time. Therefore $J_1$, $J_2$ are square-free Veronese ideals and we have the case $(b)$.
	
\end{description}

\item [Case (iii)]
Let $\mid G(J_3)\mid=2$. Then by Lemmas \ref{L5}, \ref{L7}, we have $\gcd(J_1)=y_3$, $\gcd(J_2)=1$ or $\gcd(J_1)=1=\gcd(J_2)$. If $\gcd(J_1)=y_3$, $\gcd(J_2)=1$, then we can assume that $G(J_3)=\{y_3y_4y_5,y_3y_4y_6\}$. Since $\gcd(J_2)=1$, by Proposition \ref{P2} $J_2$ is square-free Veronese ideal or almost Veronese ideal. If $J_2$ is square-free Veronese ideal, then $y_2y_5y_6,y_3y_4y_5\in J$ but $y_3y_5y_6$ or $y_4y_5y_6$ are not elements of $J$ and this is a contradiction. Let $J_2$ be an almost square-free Veronese ideal and we assume that $y_5y_6$ is the only element which is not in $J_2$.
In this case by change of variables we have $J_3=0$ and $J_1=J_2$ are almost square-free Veronese ideals and and this is the case $(c)$.
If $y_4y_5$ is the only element which is not in $J_2$, then $y_3y_4y_5, y_2y_4y_6$ are elements of $J$ but $y_2y_4y_5$ or $y_4y_5y_6$ are not elements of $J$ and this is a contradiction. Also, if $y_4y_6$ is the only element which is not in $J_2$, then again $J$ is not matroidal and this is a contradiction.
Now we can assume that $J_3=0$. If $\gcd(J_1)=y_3$, then by Lemmas \ref{L5}, \ref{L7} we have
$\gcd(J_2)=1$ or $\gcd(J_2)=y_3$. If $\gcd(J_2)=1$, then $y_1y_3y_5$ and $y_2y_iy_j$ are elements of $J$ for some $i,j=4,5,6$, but $y_1y_iy_j$ or $y_3y_iy_j$ are not elements of $J$ and this is a contradiction. Therefore $\gcd(J_2)=y_3$ and this is the case $(c)$. Also, if $\gcd(J_1)=1$ then $\gcd(J_2)=1$. If $J_1\neq J_2$ are almost square-free Veronese ideals, then again by using the above argument $J$ is not matroidal and this is a contradiction.
Therefore $J_1=J_2$ are square-free Veronese ideals or almost square-free Veronese ideals.

\item[Case (iv)]
Let $\mid G(J_3)\mid=1$. Then by Lemmas \ref{L5}, \ref{L7}, we have $\gcd(J_1)=1=\gcd(J_2)$. Therefore by Proposition \ref{P2} $J_1$ and $J_2$ are square-free Veronese ideals or almost Veronese ideals. By choosing one element from $J_1$ and the only element from $J_3$, we have $\mid G(J_3)\mid\geq 2$. This is a contradiction.
\end{description}
\end{proof}



\end{document}